\documentclass[11pt,reqno]{amsart}

\usepackage[T1]{fontenc}
\usepackage[utf8]{inputenc}
\usepackage[english]{babel}
\usepackage{amsmath,amssymb,amsthm,mathtools}
\usepackage[margin=2.8cm]{geometry}
\usepackage[colorlinks=true,citecolor=blue,linkcolor=blue,urlcolor=blue]{hyperref}
\usepackage[nameinlink,capitalise,noabbrev]{cleveref}

\newtheorem{theorem}{Theorem}[section]
\newtheorem{proposition}[theorem]{Proposition}
\newtheorem{lemma}[theorem]{Lemma}
\newtheorem{corollary}[theorem]{Corollary}
\theoremstyle{definition}
\newtheorem{example}[theorem]{Example}
\theoremstyle{remark}

\newcommand{\tr}{\operatorname{tr}}

\title[Odd cycles in integral graphs]{Spectral asymmetry and odd cycles in integral graphs}

\author{Rodrigo O. Braga}
\author{Jean Carlo Moraes}
\author{Matheus C. Santos}
\thanks{Instituto de Matem\'atica e Estat\'istica, Universidade Federal do Rio Grande do Sul, Porto Alegre, Brazil.}


\begin{document}

\begin{abstract}
We prove that the triangle $C_3$ is the only integral unicyclic graph with an odd cycle. The main tool is a spectral inequality that does not require integrality: if $G$ is a connected non-bipartite graph and $G-e$ is bipartite for some edge $e$, then the eigenvalues $\lambda_1\ge\cdots\ge\lambda_n$ of $G$ satisfy $\sum_{i}(\lambda_i+\lambda_{n+1-i})^+\le 2$ and $0<\lambda_1+\lambda_n\le 1$. It follows from comparing $A(G)$ with the signed adjacency matrix obtained by reversing the sign of $e$, which is similar to $-A(G)$, through a common rank-one perturbation and Weyl's monotonicity theorem. For integral graphs in this class, an arithmetic lemma then gives $\lambda_n=1-\lambda_1$ and at least $\lambda_1-1\ge1$ triangles, which yields the classification. We also prove that an integral graph of odd girth $g$ has at least $(g-1)!/2$ cycles of length $g$, with equality for the Petersen graph when $g=5$, and we derive consequences for quantum periodicity and for linearized evolutionary games on networks.
\end{abstract}

\maketitle

\section{Introduction}

For a finite simple graph $G$, let $A(G)$ denote its adjacency matrix. We say that $G$ is \emph{integral} if all eigenvalues of $A(G)$ are integers. The problem of characterising integral graphs dates back to Harary and Schwenk~\cite{HS}. Despite considerable progress, a complete characterization remains out of reach, even for restricted classes such as trees and unicyclic graphs, that is, connected graphs containing exactly one cycle. Integrality also has a dynamical meaning: it makes the continuous-time quantum walk $e^{-\mathrm{i}tA(G)}$ periodic, a property closely tied to perfect state transfer~\cite{Godsil,SSS}, while the adjacency spectrum determines the linearized frequencies of evolutionary games on networks~\cite{MM,MMS}; see \cref{sec:applications}.

Since cycle parity determines spectral symmetry, a natural starting point is the study of unicyclic graphs, whose unique cycle is either even, making the graph bipartite, or odd. For these graphs, Omidi~\cite{Omidi} proved that an integral unicyclic graph with no zero eigenvalue is $C_3$ or $C_6$. Using the eigenvalue location algorithm of~\cite{BRT}, Braga, Del-Vecchio and Rodrigues~\cite{BDR} constructed three infinite families of integral unicyclic graphs, all with a cycle of length four, and proposed two conjectures on integral generalized sun graphs (cycles with pendant paths attached). These conjectures were settled in~\cite{BMS}: no integral graph has a pendant path with three or more edges, and there are infinitely many integral generalized sun graphs whose cycle is a hexagon. In all known examples other than $C_3$ the cycle has even length, and the question of whether the cycle of an integral generalized sun graph that is not a cycle must have even length was left open in~\cite{BMS}.

We answer this question affirmatively, not only for generalized sun graphs but for all unicyclic graphs, without any assumption on the trees attached to the cycle.

\begin{theorem}\label{thm:main}
Let $G$ be a unicyclic graph whose unique cycle has odd length. If $G$ is integral, then $G\cong C_3$.
\end{theorem}

In other words, every integral unicyclic graph other than $C_3$ is bipartite, and therefore has a spectrum symmetric about zero \cite{BH}.

Throughout, the eigenvalues of a real symmetric matrix $X$ of order $n$ are listed with multiplicity in nonincreasing order, $\lambda_1(X)\ge\cdots\ge\lambda_n(X)$, and $\lambda_i(G)=\lambda_i(A(G))$. We write $t^+=\max\{t,0\}$ for the positive part of a real number $t$. The proof of \cref{thm:main} rests on the following inequality, which does not require integrality and may be of independent interest.

\begin{theorem}[Spectral asymmetry]\label{thm:asym}
Let $G$ be a connected non-bipartite graph on $n$ vertices having an edge $e$ such that $G-e$ is bipartite, and define $d_i=\lambda_i(G)+\lambda_{n+1-i}(G)$ for $1\le i\le n$. Then
\begin{equation}\label{eq:budget}
\sum_{i=1}^n d_i^+\le 2,\qquad \sum_{i=1}^n|d_i|\le 4,\qquad 0<\lambda_1(G)+\lambda_n(G)\le 1 .
\end{equation}
This applies to every unicyclic graph whose unique cycle has odd length, since removing any edge of the cycle leaves a tree. All three upper bounds are attained by $C_3$, whose spectrum is $\{2,-1,-1\}$.
\end{theorem}

A graph is bipartite if and only if its spectrum is symmetric about zero, that is, all $d_i$ vanish~\cite{BH}. For a connected non-bipartite graph, the Perron--Frobenius theorem gives $d_1=d_n=\lambda_1+\lambda_n>0$. \Cref{thm:asym} shows that if the removal of a single edge makes $G$ bipartite, this asymmetry is bounded: $\lambda_1+\lambda_n\le 1$ and $\sum_i d_i^+\le 2$, regardless of the order of $G$ and of its number of cycles. The quantity $\lambda_1+\lambda_n$ for graphs without short odd cycles has recently been studied in~\cite{ATV}.


\section{Spectral asymmetry and the proof of the main theorem}\label{sec:proof}

We write $X\succeq Y$ when $X-Y$ is positive semidefinite.

\begin{lemma}\label{lem:weyl}
Let $X$, $Y$, $M$ be real symmetric matrices of order $n$ with $M\succeq X$ and $M\succeq Y$. Then
$\sum_{i=1}^n\bigl(\lambda_i(X)-\lambda_i(Y)\bigr)^+\le\tr M-\tr Y$.
\end{lemma}

\begin{proof}
By Weyl's monotonicity theorem~\cite[Section~4.3]{HJ}, $\gamma_i=\lambda_i(M)$ satisfies $\gamma_i\ge\lambda_i(X)$ and $\gamma_i\ge\lambda_i(Y)$, so $\bigl(\lambda_i(X)-\lambda_i(Y)\bigr)^+\le\gamma_i-\lambda_i(Y)$ for every $i$. Summing over $i$ gives the claim, since the trace is the sum of the eigenvalues.
\end{proof}

We shall use the following standard facts for a connected graph $G$ with adjacency matrix $A$; see~\cite{BH}.
\begin{enumerate}
\item[(P1)] $\tr A=0$, and $\tr A^3=6t(G)$, where $t(G)$ is the number of triangles of $G$.
\item[(P2)] (Perron--Frobenius) $\lambda_1(G)$ is the spectral radius of $A$, so $\lambda_1(G)\ge|\lambda_n(G)|$. Moreover $\lambda_n(G)=-\lambda_1(G)$ if and only if $G$ is bipartite.
\item[(P3)] $\lambda_1(G)=\max_{z\ne 0}z^{T}Az/z^{T}z$.
\item[(P4)] If $H$ is a proper subgraph of $G$, then $\lambda_1(H)<\lambda_1(G)$.
\end{enumerate}

\begin{proof}[Proof of \cref{thm:asym}]
Write $e=v_1v_2$ and $A=A(G)$. Choose a proper $2$-colouring $s_v\in\{1,-1\}$ of the bipartite graph $G-e$, so that $s_us_v=-1$ whenever $uv$ is an edge of $G-e$, and set $D=\operatorname{diag}(s_v)$. Then $s_{v_1}=s_{v_2}$, since otherwise $s$ would be a proper $2$-colouring of $G$, contradicting the assumption that $G$ is not bipartite.

Put $B=-DAD$, so that $B_{uv}=-s_us_vA_{uv}$. If $uv$ is an edge of $G-e$, then $B_{uv}=A_{uv}=1$. On the other hand, $B_{v_1v_2}=B_{v_2v_1}=-1$, while all other entries of $A$ and $B$ are zero. Thus
\begin{equation}\label{eq:B}
B=A-2\bigl(e_{v_1}e_{v_2}^{T}+e_{v_2}e_{v_1}^{T}\bigr),
\end{equation}
where $e_v$ denotes the standard basis vector corresponding to $v$. In other words, $B$ is the adjacency matrix of the signed graph obtained from $G$ by assigning a negative sign to the edge $e$.
Since $D^2=I$, the matrix $B$ is similar to $-A$, and hence
\begin{equation}\label{eq:beta}
\lambda_i(B)=-\lambda_{n+1-i}(G)\qquad(1\le i\le n).
\end{equation}

Now let $w_-=e_{v_1}-e_{v_2}$ and $w_+=e_{v_1}+e_{v_2}$. Using \eqref{eq:B}, we obtain
\[
M:=A+w_-w_-^{T}=A+e_{v_1}e_{v_1}^{T}+e_{v_2}e_{v_2}^{T}-e_{v_1}e_{v_2}^{T}-e_{v_2}e_{v_1}^{T}=B+w_+w_+^{T}.
\]
Since $w_\pm w_\pm^{T}\succeq 0$, it follows that $M\succeq A$ and $M\succeq B$. Moreover, both $A$ and $B$ have zero diagonal, so $\tr A=\tr B=0$ and $\tr M=\tr A + \|w_-\|^2=2$. Applying \cref{lem:weyl} with $X=A$, $Y=B$, and using \eqref{eq:beta}, gives
\[
\sum_{i=1}^n d_i^+=\sum_{i=1}^n\bigl(\lambda_i(G)-\lambda_i(B)\bigr)^+\le\tr M-\tr B=2 .
\]
Furthermore, $\sum_i d_i=2\tr A=0$, so the positive and negative parts of the $d_i$ have the same total size, and $\sum_i|d_i|=2\sum_i d_i^+\le 4$. Since $G$ is connected and not bipartite, (P2) gives $d_1=d_n=\lambda_1(G)+\lambda_n(G)>0$, hence $2(\lambda_1+\lambda_n)=d_1+d_n\le\sum_id_i^+\le 2$. For $C_3$ we have $d_1=d_3=1$ and $d_2=-2$, so all three upper bounds are attained.
\end{proof}

The next lemma, which involves no graph theory, shows how integrality turns this bound into a constraint on the sum of cubes.

\begin{lemma}\label{lem:arith}
Let $n\ge 3$ and let $a_1\ge\cdots\ge a_n$ be integers such that
\begin{equation}\label{eq:hyp}
r:=a_1\ge 2,\qquad \sum_{i=1}^na_i=0,\qquad a_n>-r,
\end{equation}
and, writing $\delta_i=a_i+a_{n+1-i}$, assume that $\sum_{i=1}^n\delta_i^+\le 2$. Then $a_n=1-r$, and
\begin{equation}\label{eq:cubes}
\sum_{i=1}^na_i^3\ge 6(r-1).
\end{equation}
\end{lemma}

\begin{proof}
Note that $\delta_i=\delta_{n+1-i}$ for every $i$. Since $r+a_n$ is a positive integer, $\delta_1=\delta_n\ge 1$, and the hypothesis $\sum_i\delta_i^+\le 2$ forces
\[
\delta_1=\delta_n=1,\qquad \delta_i\le 0\quad(2\le i\le n-1).
\]
In particular $a_n=1-r$. Moreover $\sum_i\delta_i=2\sum_ia_i=0$, so $\sum_{i=2}^{n-1}\delta_i=-2$.

The interior indices $2\le i\le n-1$ split into reflected pairs $\{j,n+1-j\}$ with $2\le j<(n+1)/2$, each contributing $2\delta_j$ to this sum, and, when $n$ is odd, the central index $c=(n+1)/2$, contributing $\delta_c=2a_c$. All these contributions are even nonpositive integers adding up to $-2$, so exactly one of them equals $-2$ and the others vanish. Thus either (a) $n$ is odd, $a_c=-1$, and every interior reflected pair has $\delta_j=0$; or (b) there is exactly one $j$ with $2\le j<(n+1)/2$ and $\delta_j=-1$, every other interior reflected pair has $\delta=0$, and $a_c=0$ if $n$ is odd. A reflected pair with $\delta_j=0$ consists of two opposite integers and contributes nothing to $\sum_ia_i^3$. In case (b) put $s=a_j$, so that $a_{n+1-j}=-s-1$. The ordering $a_j\ge a_{n+1-j}$ gives $s\ge -1/2$, hence $s\ge 0$, and $a_{n+1-j}\ge a_n=1-r$ gives $s\le r-2$. Case (a) is the same computation with $s=0$, since $0^3+(-1)^3=s^3+(-s-1)^3$ for $s=0$. In both cases, with $0\le s\le r-2$,
\[
\sum_{i=1}^na_i^3=r^3+(1-r)^3+s^3+(-s-1)^3=3\bigl(r(r-1)-s(s+1)\bigr).
\]
Since $s(s+1)\le(r-2)(r-1)$, we obtain $\sum_ia_i^3\ge 3\bigl(r(r-1)-(r-2)(r-1)\bigr)=6(r-1)$.
\end{proof}

Combining \cref{thm:asym} with \cref{lem:arith} gives the following triangle bound, of which \cref{thm:main} is a special case.

\begin{theorem}\label{thm:triangles}
Let $G$ be a connected integral non-bipartite graph such that $G-e$ is bipartite for some edge $e$. Set $\rho=\lambda_1(G)$ and let $t(G)$ be its number of triangles. Then $\lambda_n(G)=1-\rho$ and $t(G)\ge\rho-1\ge1$. Moreover, $t(G)=1$ if and only if $G\cong C_3$.
\end{theorem}
\begin{proof}
Write $A=A(G)$. Since $G$ is non-bipartite, it contains an odd cycle of some length $g$. Let $z$ be the indicator vector of the vertices of this cycle. Then $z^Tz=g$ and $z^TAz\ge2g$, so (P3) gives $\rho\ge2$.

The eigenvalues are integers because $G$ is integral. Their sum is zero by (P1), and $\lambda_n(G)>-\rho$ by (P2), since $G$ is connected and non-bipartite. Finally, \cref{thm:asym} gives $\sum_i(\lambda_i(G)+\lambda_{n+1-i}(G))^+\le2$. Thus all hypotheses of \cref{lem:arith} hold with $a_i=\lambda_i(G)$ and $r=\rho$. It follows that $\lambda_n(G)=1-\rho$ and $\sum_i\lambda_i(G)^3\ge6(\rho-1)$. Using (P1), we obtain $6t(G)=\tr A^3=\sum_i\lambda_i(G)^3\ge6(\rho-1)$, hence $t(G)\ge\rho-1\ge1$.

If $t(G)=1$, the preceding inequality gives $\rho\le2$, so $\rho=2$.   {Since $G$ contains a triangle, whose spectral radius is $2$, (P4) implies that this triangle cannot be a proper subgraph of $G$. 
Thus $G\cong C_3$. The converse is immediate.}
\end{proof}

\begin{proof}[Proof of \cref{thm:main}]
Since its unique cycle is odd, $G$ is non-bipartite, 
and deleting any edge of this cycle leaves a tree, which is bipartite. Therefore, $G$ satisfies the hypotheses of \cref{thm:triangles}. As $G$ is unicyclic, $t(G)\le1$, whereas the theorem gives $t(G)\ge1$. Hence $t(G)=1$ and consequently $G\cong C_3$.
\end{proof}

\section{Integral graphs with odd girth}\label{sec:remarks}

We now prove a divisibility property of spectral moments valid for all integral graphs. Recall that the \emph{odd girth} of a non-bipartite graph is the length of its shortest odd cycle, and write $c_g(H)$ for the number of cycles of length $g$ in a graph $H$.

\begin{proposition}\label{prop:div}
Let $H$ be an integral graph with odd girth $g$. Then $(g-1)!$ divides $2c_g(H)$.
\end{proposition}

\begin{proof}
Every closed walk of odd length $\ell$ contains an odd cycle of length at most $\ell$: if the walk has a repetition other than the coincidence of its initial and final vertices, splitting it at the repetition gives two shorter closed walks whose lengths add up to $\ell$, one of which is odd; iterating, one reaches an odd closed walk without repeated vertices other than its endpoints, that is, an odd cycle. Consequently $\tr A(H)^k=0$ for odd $k<g$, and a closed walk of length $g$ has no repeated vertex other than its endpoints, so it traverses a $g$-cycle exactly once. Each $g$-cycle gives $2g$ such walks (choice of starting vertex and orientation), so $\tr A(H)^g=2g\,c_g(H)$.

Write $g=2m+1$ and $Q(x)=\prod_{k=-m}^{m}(x-k)=x\prod_{k=1}^m(x^2-k^2)$. For every integer $a$, $Q(a)$ is a product of $g$ consecutive integers, hence divisible by $g!$. Moreover $Q$ is odd and monic of degree $g$, so $Q(x)=x^g+\sum_{j=1}^mb_jx^{g-2j}$ with integers $b_j$ and odd exponents $1\le g-2j<g$. Summing over the integer eigenvalues $\theta_1,\dots,\theta_N$ of $A(H)$,
\[
\sum_{i=1}^NQ(\theta_i)=\tr A(H)^g+\sum_{j=1}^mb_j\tr A(H)^{g-2j}=2g\,c_g(H).
\]
Since the left-hand side is divisible by $g!$, dividing by $g$ yields the claim.
\end{proof}

Since $H$ has at least one cycle of length $g$, the quotient $2c_g(H)/(g-1)!$ is a positive integer. This gives the following lower bound.

\begin{corollary}\label{cor:many}
Let $H$ be an integral graph with odd girth $g$. Then
\[
c_g(H)\ge\frac{(g-1)!}{2}.
\]
In particular, an integral graph with odd girth $5$ has at least $12$ pentagons, and one with odd girth $7$ has at least $360$ heptagons.
\end{corollary}

For $g=5$, integrality forces the number of pentagons in a triangle-free graph to be a multiple of $12$.

\begin{example}\label{ex:petersen}
The Petersen graph $\mathcal P$ has spectrum $\{3,\,1^{[5]},\,(-2)^{[4]}\}$, so it is integral. It has girth $5$, hence odd girth $5$, and it contains exactly $12$ pentagons. Thus $c_5(\mathcal P)=12=4!/2$, and equality holds in \cref{cor:many}, which is therefore sharp for $g=5$.
\end{example}

\Cref{cor:many} also yields a bound in terms of the cyclomatic number $k=|E(H)|-|V(H)|+\omega(H)$, where $\omega(H)$ is the number of connected components. The cycle space $\mathcal C(H)$ over $\mathbb F_2$ has dimension $k$, and the parity map $F\mapsto|F|\bmod 2$ is linear on it, since addition corresponds to symmetric difference of edge sets. If $H$ is non-bipartite this map is nonzero, so exactly $2^{k-1}$ elements of $\mathcal C(H)$ have odd cardinality. Distinct odd cycles are distinct elements of $\mathcal C(H)$, hence $c_g(H)\le 2^{k-1}$, and \cref{cor:many} gives the following result.

\begin{corollary}\label{cor:fewcycles}
Every integral non-bipartite graph with cyclomatic number at most $4$ contains a triangle. More generally, an integral graph with cyclomatic number $k$ and odd girth $g$ satisfies $2^k\ge(g-1)!$.
\end{corollary}

\begin{proof}
The general inequality follows from $(g-1)!/2\le c_g(H)\le 2^{k-1}$. If $k\le4$ and $g\ge5$, then $24\le(g-1)!\le2^k\le16$, a contradiction. Hence $g=3$.
\end{proof}

For a unicyclic graph whose unique cycle has odd length $g$, \cref{prop:div} with $c_g=1$ gives $(g-1)!\mid2$, hence $g=3$, which is an alternative proof that the cycle in \cref{thm:main} is a triangle. Divisibility alone does not exclude trees attached to the triangle: attaching a pendant vertex to $C_3$ gives a non-integral graph with $2!\mid2c_3$. The spectral argument of \cref{sec:proof} excludes them through $\lambda_1(G)=2$. Likewise, \cref{cor:fewcycles} does not force $C_3$: the complete graph $K_4$ is integral, with cyclomatic number $3$ and four triangles. Finally, the divisibility statement fails for even cycles, as $C_4$ shows ($3!\nmid2$), because even moments also count closed walks that retrace edges; infinite families of integral unicyclic graphs with cycle length $4$ or $6$ are known~\cite{BDR,BMS}.

\section{Dynamical applications}\label{sec:applications}

For the continuous-time quantum walk $U(t)=e^{-\mathrm{i}tA(G)}$, each eigenvector of $A(G)$ evolves by the phase factor $e^{-\mathrm{i}t\lambda_i(G)}$. Following Godsil~\cite{Godsil}, we say that $G$ is \emph{periodic} if $U(T)$ is a scalar multiple of the identity for some $T>0$. In zero-sum rock--paper--scissors on a network, each vertex is a player whose mixed strategy follows the replicator equation with payoffs summed over its neighbours. Linearizing at the homogeneous equilibrium, where every player uses $(1/3,1/3,1/3)$, gives angular frequencies $\omega_i=|\lambda_i(G)|/\sqrt3$~\cite[Eq.~(2) and Section~4]{GSB}.   {Consequently}, all linear modes 
  {have} a common period precisely when the nonzero eigenvalues of $A(G)$ have rational ratios. For non-bipartite graphs, both this 
  {condition} and periodicity reduce to integrality.

\begin{proposition}\label{prop:commens}
For a connected non-bipartite graph $G$, the following are equivalent:
\textup{(i)} $G$ is periodic;
\textup{(ii)} all ratios of nonzero eigenvalues of $A(G)$ are rational;
\textup{(iii)} $G$ is integral.
\end{proposition}

\begin{proof}
If $G$ is integral, then $U(2\pi)=I$, so (iii) implies (i). Assume (i), with $U(T)=\gamma I$. Then $T(\lambda_1(G)-\lambda_i(G))\in2\pi\mathbb Z$ for every $i$, so $\lambda_1(G)-\lambda_i(G)=k_ic$, with $c=2\pi/T$ and $k_i\in\mathbb Z$. Summing over $i$ and using $\tr A=0$ gives $n\lambda_1(G)=c\sum_ik_i$, hence every eigenvalue lies in $c\,\mathbb Q$, which gives (ii).

Assume (ii) and let $\rho=\lambda_1(G)>0$, so that every nonzero eigenvalue lies in $\mathbb Q\rho$. Let $d$ be the degree of the minimal polynomial of $\rho$ over $\mathbb Q$. Its roots are eigenvalues of $A(G)$, since the characteristic polynomial has integer coefficients, and they are nonzero, so they have the form $c_j\rho$, with $c_j\in\mathbb Q\setminus\{0\}$. Their product is rational, so $\rho^d\in\mathbb Q$. Every conjugate $\beta=c_j\rho$ of $\rho$ therefore satisfies $\beta^d=\rho^d$, so $c_j^d=1$ and $\beta=\pm\rho$. If $\rho$ were irrational, then $-\rho$ would be a conjugate of $\rho$ and hence an eigenvalue, contradicting $\lambda_n(G)>-\rho$ from (P2). Thus $\rho$ is rational. Since $\rho$ is an eigenvalue of an integer matrix, it is an algebraic integer, and hence $\rho\in\mathbb Z$.
Every nonzero eigenvalue is 
  {therefore} a rational algebraic integer, hence an integer.   {Thus $G$ is integral, proving (iii)}. \end{proof}

The equivalence of (i) and (iii) also follows from Godsil's characterization of periodic graphs~\cite[Theorem~3.1 and Corollary~3.3]{Godsil}. Thus, \cref{thm:main} 
  {implies that the triangle is the only unicyclic graph with an odd cycle that is periodic or has commensurable linear frequencies.}
For $C_3$, these   {frequencies} are $1/\sqrt3$ and $2/\sqrt3$, so all linear modes repeat after time $2\pi\sqrt3$. These conclusions concern the linearized dynamics; the nonlinear game can exhibit chaos even on $C_3$~\cite{GSB}.

\section*{Declaration of generative AI and AI-assisted technologies in the writing process}

During the preparation of this work the authors used Claude (Anthropic) as a discussion partner for brainstorming and exploring ideas, to check proofs and the exposition for errors, and to edit the language and the \LaTeX{} source. After using this tool, the authors reviewed and edited the content as needed, verified all mathematical statements and proofs, and take full responsibility for the content of the publication.


\begin{thebibliography}{99}

\bibitem{ATV}
A. Abiad, V. Taranchuk, and T. van Veluw, \emph{On the sum of the largest and smallest eigenvalues of odd-cycle free graphs}, Electron. J. Combin. \textbf{33}(2) (2026), \#P2.31.

\bibitem{BDR}
R.~O. Braga, R.~R. Del-Vecchio, and V.~M. Rodrigues, \emph{Integral unicyclic graphs}, Linear Algebra Appl. \textbf{614} (2021), 281--300.

\bibitem{BMS}
R.~O. Braga, J.~C. Moraes, and M.~C. Santos, \emph{Pendant paths and integral generalized sun graphs}, arXiv:2609.28754 (2026).

\bibitem{BRT}
R.~O. Braga, V.~M. Rodrigues, and V. Trevisan, \emph{Locating eigenvalues of unicyclic graphs}, Appl. Anal. Discrete Math. \textbf{11} (2017), 273--298.

\bibitem{BH}
A.~E. Brouwer and W.~H. Haemers, \emph{Spectra of Graphs}, Universitext, Springer, New York, 2012.

\bibitem{Godsil}
C. Godsil, \emph{Periodic graphs}, Electron. J. Combin. \textbf{18}(1) (2011), \#P23.

\bibitem{GSB}
C. Griffin, J. Semonsen, and A. Belmonte, \emph{Generalized Hamiltonian dynamics and chaos in evolutionary games on networks}, Physica A \textbf{597} (2022), 127281.

\bibitem{HS}
F. Harary and A.~J. Schwenk, \emph{Which graphs have integral spectra?}, in: Graphs and Combinatorics (R.~A. Bari and F. Harary, eds.), Lecture Notes in Math. 406, Springer, Berlin, 1974, pp.~45--51.

\bibitem{HJ}
R.~A. Horn and C.~R. Johnson, \emph{Matrix Analysis}, 2nd ed., Cambridge University Press, Cambridge, 2013.

\bibitem{MM}
D. Madeo and C. Mocenni, \emph{Game interactions and dynamics on networked populations}, IEEE Trans. Automat. Control \textbf{60} (2015), 1801--1810.

\bibitem{MMS}
C. Mocenni, J.~C. Moraes, and M.~C. Santos, \emph{On the structure and stability of boundary mixed steady states in evolutionary games on networks}, arXiv:2606.20688 (2026).

\bibitem{Omidi}
G.~R. Omidi, \emph{On integral graphs with few cycles}, Graphs Combin. \textbf{25} (2009), 841--849.

\bibitem{SSS}
N. Saxena, S. Severini, and I.~E. Shparlinski, \emph{Parameters of integral circulant graphs and periodic quantum dynamics}, Int. J. Quantum Inf. \textbf{5} (2007), 417--430.

\end{thebibliography}
\end{document}